\documentclass[reqno,centertags,a4paper,12pt]{amsart}
\usepackage{amssymb}
\usepackage{amsthm}
\usepackage{eucal}
\usepackage{color, xcolor}
\usepackage[english]{babel}
\usepackage{graphicx}
\usepackage{hyperref}

\newcommand{\R}{\mathbb R}

\newcommand{\Z}{\mathbb Z}
\newcommand{\T}{\mathbb T}

\numberwithin{equation}{section}
\newtheorem{theorem}{Theorem}[section]
\newtheorem{proposition}[theorem]{Proposition}
\newtheorem{remark}[theorem]{Remark}
\newtheorem{lemma}[theorem]{Lemma}

\begin{document}
\title[Higher order 
 nonlinear Schr\"odinger equation]{Global well-posedness for the higher order non-linear Scrh\"odinger equation in modulations spaces}

\address{Universidade Federal do Rio de Janeiro}
\author{X. Carvajal}
\author{P. Gamboa}
\author{R. Santos}
\email{carvajal@im.ufrj.br}
\email{pgamboa@im.ufrj.br}
\email{raphaelsantos@macae.ufrj.br}


\maketitle

\begin{abstract} We consider the
 initial value problem (IVP) associated to a higher order nonlinear Schr\"odinger (h-NLS)  equation
\begin{equation*} 
\partial_{t}u+ia
\partial^{2}_{x}u+ b\partial^{3}_{x}u=2ia|u|^{2}u+6b
|u|^{2}\partial_{x}u,
 \quad x,t \in \R, 
\end{equation*}
for given data in the modulation space $M_s^{2,p}(\R)$. Using ideias of Killip, Visan,  Zhang, Oh, Wang, we prove that the IVP associated to the h-NLS
equation is globally well-posed in the modulation spaces $M^{s,p}$ for $s\geq\frac14$ and $p\geq2$.

\end{abstract}


\section{Introduction}

In this work we  consider the initial value problem (IVP) associated to a higher order nonlinear Schr\"odinger (h-NLS)
 equation
\begin{equation}\label{xhonse}
\begin{cases} 
\partial_{t}u+ia 
\partial^{2}_{x}u+ b \partial^{3}_{x}u=2ia|u|^{2}u+6b
|u|^{2}\partial_{x}u,
 \quad x,t \in \R,\\
  u(x,0) = u_0(x),
 \end{cases}
\end{equation}
where $a,b \in \R$ and $u = u(x, t)$ is a complex valued function.
 
The main objective of this work is to investigate the global well-posedness issues of the IVP \eqref{xhonse} in the modulation spaces.  




In recent time, well-posedness of the IVPs associated to the nonlinear dispersive equations are being studied in some other scales of the function spaces than the usual $L^2$ based Sobolev spaces $H^s(\R)$ viz., the Fourier-Lebesgue spaces $\mathcal{F}L^{s,p}(\R)$ with norm 
 $$
 \|u\|_{\mathcal{F}L^{s,p}} =\|\langle \xi\rangle^s \widehat{u}(\xi)\|_{L^p},
 $$
  and modulation spaces  $M_s^{r,p}(\R)$  with norm given by \eqref{def-2} (below). More specifically we mention the local well-posedness result for the modified Korteweg-de Vries (mKdV) equation in $\mathcal{F}L^{s,p}(\R)$ for $s\geq \frac1{2p}$ with $2\leq p<4$ obtained in \cite{Gr-04} and its improvement for the same range of $s$ with $2\leq p <\infty$ in \cite{GV-09}. For other discussion about these results we refer the readers to \cite{Oh-Wang} where the authors considered the IVP associated to the complex-valued mKdV equation in the modulation spaces  $M_s^{2,p}(\R)$ and proved local well-posedness for $s\geq \frac14$ with $2\leq p <\infty$. Quite recently, Oh and Wang \cite{OW-20} introduced a new function space $H\!M^{\theta,p}$ whose norm is given by the $\ell^p$-sum of the modulated $H^{\theta}$-norm of a given function and agrees with the modulation space $M^{2,p}(\R)$ on the real line and Fourier-Lebesgue space $\mathcal{F}L^p(\T)$ on the circle. The authors in \cite{OW-20} proved that the cubic NLS is globally well-posed in $M^{2,p}(\R)$ for any $p<\infty$ and the normalized cubic NLS is globally well-posed in $\mathcal{F}L^p(\T)$ for any $p<\infty$. As far we know, there are no known results about the global well-posedness issues for the IVP \eqref{xhonse} for given data in the modulation spaces.
 
Our interest here is in  addressing the well-posedness issues for the h-NLS equation \eqref{xhonse}  with given  data in the modulation spaces. We obtain the global well-posedness result for the IVP \eqref{xhonse} in the same spirit to that for the complex mKdV equation  \cite{Oh-Wang}.



The h-NLS is a particular case of the more general equation (honse equation)
\begin{equation}\label{honse}
\begin{cases} 
\partial_{t}u+ia 
\partial^{2}_{x}u+ b \partial^{3}_{x}u+ic_1|u|^{2}u+c_2
|u|^{2}\partial_{x}u+du^2\partial_x\overline{u}=0,
 \quad x,t \in \R,\\
  u(x,0) = u_0(x),
 \end{cases}
\end{equation}
when $d=0$, $c_1=-2a$, $c_2=-6b$. The honse equation \eqref{honse}, which is a   mixed model of complex Korteweg-de Vries (KdV) and Schr\"{o}dinger type was
 proposed by  Hasegawa and Kodama in \cite{[H-K]} and \cite{[Ko]} to describe
 the  nonlinear propagation of pulses in optical fibers. The IVP (\ref{honse})
 has also
been studied by several authors in recent years. Taking $a, b, c_1, c_2$ and $d$
 as real constants, Laurey
\cite{[C2]} proved that the IVP (\ref{honse}) for given
data in $H^s(\R)$ is locally well-posed when $s>\frac34$ and
globally well-posed when $s\geq 1$. Later, using the techniques
developed by Kenig, Ponce and Vega \cite{[KPV1]}, Staffilani
\cite{[G]} improved the result in \cite{[C2]} by showing that
the IVP  (\ref{honse}) is locally well-posed in
$H^s(\R),\,s\geq \frac14$. Using the method of almost conserved quantities and the I-method introduced by Colliander et. al. \cite{CKSTT},
Carvajal \cite{XC-06} proved the sharp global well-posedness of 
IVP associated to (\ref{honse}) in $H^s(\R)$ for $s >
\frac14$. The IVP  ~(\ref{honse}) when $a$ and $b$ are
functions of $t\in [-T_0, T_0]$ for some $T_0>0$ and $b(t) \neq 0$
for all $t\in [-T_0, T_0]$ has also been a matter of study (see for
instance \cite{[C2]},  \cite{CA1} and \cite{CL1}). In \cite{CP} the authors proved the local well-posedness to the IVP \eqref{honse} in the modulation spaces $M^{s,p}$ for $s\geq\frac14$ and $p\geq2$.

The IVP \eqref{honse} posed on the circle $\T$ is also studied in the literature, see for instance \cite{Takaoka} and references therein.


As far as we know, no work is available in the literature that deals with the global well-posedness issues of the IVP \eqref{honse} for given data in the modulation spaces. Motivated by the recent works in \cite{Oh-Wang} and \cite{CP}, our interest in this work is to address this issue. In fact, we prove the global well-posedness result for the IVP \eqref{honse} for given data in the modulation space $M_s^{2,p}(\R)$, whenever $s\geq\frac14$. This is the content the following theorem which is the main result of this work.

\begin{theorem}\label{main-th2}
For given $s\geq \frac14$ and $2\leq p<\infty$, the IVP \eqref{xhonse} is globally well-posed in the modulation space $M_s^{2,p}(\R) $.
\end{theorem}


We present the organization of this work. 
 In Section \ref{sec-3} we introduce the function spaces, their properties  and record some preliminary results. Section \ref{sec-4} is devoted to properties of traces, multiplication of operators and we derive the key conservation of the perturbation determinant $\alpha(k,u)$ (see definition \eqref{definalpha}) that is fundamental to prove the main result of this work. In Section \ref{sec-5} we provide the proofs of an apriori estimate in modultion spaces and the main result of this paper.  We finish this section recording some principal notations that will be used throughout this work.\\

\noindent
{\textbf{Notations:}} We will use standard notations of the PDEs 
throughout this work.  We use $\widehat{f}$ to denote Fourier
transform and is defined by $\widehat{f}(\xi) = (2\pi)^{-1/2} \int
e^{-ix\xi}f(x)\,dx.$  We write $A\lesssim B$
if there exists a constant $c >0$ such that $ A \leq cB$, we also write $A \sim B$ if $A\lesssim B$ and $B\lesssim A$... etc.

\section{Function spaces and preliminary results}\label{sec-3} 
    As described in the previous section, the best global well-posedness result for the IVP \eqref{xhonse} for given data in the Sobolev space $H^s(\R)$, $s\geq 1/4$ was obtained in \cite{XC-06}, using the Fourier transform norm space $Z^{s,b}$ defined 
    for $s,b\in\R$, $Z^{s,b}$ is the Fourier transform restriction norm space introduced by Bourgain \cite{B-93}  with norm
 \begin{equation}
 \label{X-norm}
 \|u\|_{Z^{s,b}}:=\|\langle\xi\rangle^s\langle\tau-\phi(\xi)\rangle^b\widehat{u}(\xi, \tau)\|_{L^2_{\xi}L^2_{\tau}},
 \end{equation}
 where $\langle x\rangle:=1+|x|$ and $\phi(\xi) = b\xi^3+a \xi^2$ is the phase function associated to the h-NLS equation \eqref{xhonse}. We note that, for $b>\frac12 $ one has $Z^{s,b}\subset C(\R; H^s(\R))$                     and these spaces play a very important role in obtaining the well-posedness results for the IVP associated to the dispersive equations with low regularity Sobolev data.

Now, we move on to introduce modulation spaces on which we are interested to concentrate our work. For given $s\in \R$, $1\leq r, p\leq\infty$,  modulation spaces 
    $M_s^{r,p}(\R)$ are defined by \cite{Feich-1, Feich-2}   
    \begin{equation}
    \label{def-1}
      M_s^{r,p}(\R):=\{f\in \mathcal{S}'(\R) : \|f\|_{M_s^{r,p}}<\infty\},
      \end{equation}
where
\begin{equation}
    \label{def-2}\|f\|_{M_s^{r,p}}:= \|\langle n\rangle^s\|\psi_n(D)f\|_{L_x^r(\R)}\|_{\ell_n^p(\Z)},
\end{equation}
with $\psi\in \mathcal{S}(\R)$ such that
$${\mathrm{supp}}\, \psi\subset[-1, 1], \qquad \sum_{k\in\Z}\psi(\xi-k)=1,$$
and $\psi_n(D)$ is the Fourier multiplier operator with symbol
$$\psi_n(\xi):=\psi(\xi-n).$$

For given $n\geq 1$, let $P_N$ be the Littlewood-Paley projector on the frequencies $\{|\xi|\sim N\}$.

For $n\in \Z$ we define
\begin{equation}
\label{proj-2}
\widehat{\Pi_n f}(\xi):=\psi_n(\xi)\widehat{f}(\xi).
\end{equation}

For any $1\leq q\leq p\leq \infty$, from Bernstein's inequality we have the followings
\begin{equation}
\label{bern-1}
\begin{split}
&\|P_N f\|_{L_x^p}\lesssim N^{\frac1q-\frac1p}\|f\|_{L_x^q},\\
&\|\Pi_n f\|_{L_x^p}\lesssim\|f\|_{L_x^q}.
\end{split}
\end{equation}

Now, we introduce the Bourgain's type space $X_p^{s,b}$ adapted to the modulation space $M_s^{2,p}(\R)$ with norm given by
\begin{equation}
\label{bg-m}
\|f\|_{X_p^{s,b}}:= \Big(\sum_{n\in\Z}\langle n\rangle^{sp}\|\langle\tau-\xi^3\rangle^b\widehat{f}(\xi,\tau)\|^p_{L^2_{r,\xi}(\R\times[n, n+1])}\Big)^{\frac1p}\sim \|\|\Pi_nf\|_{X^{s,b}}\|_{\ell^p_n}.
\end{equation}
For $p=2$, the space $X_p^{s,b}$ simply reduces to the usual Bourgain's space $X^{s,b}$. Note that, for $b>\frac12$, one has the following inclusion
\begin{equation}
\label{sob-1}
X_p^{s,b}\subset C(\R: M_s^{2,p}(\R)).
\end{equation}

Also  the following estimates hold
\begin{equation}
\label{incl-2}
\|x_n\|_{\ell^p_n}\leq \|x_n\|_{\ell^q_n}, \qquad p\geq q\geq 1,
\end{equation}

\begin{equation}
\label{incl-2}
\|u\|_{X_p^{s,b}}\leq \|u\|_{X_q^{s,b}}, \qquad p\geq q\geq 1,
\end{equation}

\begin{equation}
\label{incl-3}
\|P_Nu\|_{X_q^{s,b}}\lesssim N^{\frac1q-\frac1p} \|P_Nu\|_{X_p^{s,b}}, \qquad p\geq q\geq 1.
\end{equation}

For a given time interval $I$, the local-in-time version  $X^{s,b}_p(I)$ of $X^{s,b}_p$ are defined with the norm
$$\|f\|_{X^{s,b}_p(I)}:=\inf\{\|g\|_{X^{s,b}_p}: g|_{I}=f\}.$$

In what follows we record some preliminary results. We start with the estimates that the unitary group satisfies in the $X^{s,b}_p$ spaces from \cite{Oh-Wang}.

\begin{lemma}\label{lem-1}
Let $s, b\in\R$ and $1\leq p<\infty$. Then for any $0<T\leq 1$ the following estimate holds
\begin{equation}
\label{est-311}
\Big\|e^{-t\partial_x^3}f\Big\|_{X^{s,b}_p([0, T])} \lesssim \|f\|_{M^{2,p}_s}.
\end{equation}
\end{lemma}

\begin{lemma}\label{lem-2}
Let $s\in\R$, $-\frac12<b'\leq 0\leq b\leq 1+b'$ and $1\leq p<\infty$. Then for any $0<T\leq 1$ the following estimate holds
\begin{equation}
\label{est-312}
\left\|\int_0^t e^{-(t-t')\partial_x^3}F(t')dt'\right\|_{X^{s,b}_p([0, T])} \lesssim T^{1+b'-b}\|F\|_{X^{s,b'}_p([0, T])}.
\end{equation}
\end{lemma}


\section{Operators and Traces}\label{sec-4}
Let $f \in S(\R)$, we define the linear operator $T$ com kernel $K_T\in L^2(\R^2)$,
\begin{equation}\label{HS}
Tf(x):=\int_{\R} K_T(x,y) f(y) dy. 
\end{equation}
Observe that $T: L^2 \to L^2$, is a bounded linear operator, in fact using Minkowsky and Cauchy-Schwartz inequalities holds $\|T\|_{L^2 \to L^2} \leq \|K_T\|_{ L^2(\R^2)}$. We define the trace  of the operator T as:
$$
\text{tr}(T)=\int_{\R} K_T(x,x)  dx.
$$
Using Fubini's Theorem in $\langle Tf, g \rangle=\int_{\R} Tf(x) \overline{g(x)} dx$, $f,g \in S(\R)$, we obtain
\begin{equation}\label{kerAdj}
K_{T^*}(x,y)=\overline{K_{T}(y,x)}.
\end{equation}
Observe that if $T_j$ has kernel $K_j$, $j=1,2$, then 
$$
T_1T_2f(x)=\int_{\R^2} K_1(x,y) K_2(y,z) f(z) dy dz=\int_{\R}\left( \int_{\R} K_1(x,y) K_2(y,z) dy\right) f(z) dz, 
$$
thus $T_1T_2$ has kernel and trace
$$
K(x,z)=\int_{\R} K_1(x,y) K_2(y,z) dy, \quad \text{tr}(T_1T_2)=\int_{\R^2} K_1(x,y) K_2(y,x) dy dx.
$$
Using Fubini's Theorem, we have
\begin{equation}\label{ConmTr}
\text{tr}(T_1T_2)=\text{tr}(T_2T_1).
\end{equation}
We also set
\begin{align}\label{definNormaoper}
\|T\|^2=\text{tr}(T T^* )=\int_{\R^2}|K_T(x,y)|^2 dx dy.
\end{align}
The operator \eqref{HS} is a Hilbert–Schmidt operator and this norm \eqref{definNormaoper} is the Hilbert–Schmidt norm.
In general if $T_j$ has kernel $K_j$, $j=1,2, \dots,n$, then
\begin{equation}\label{ProdTn}
T_1T_2\dots T_n f(x)=\int_{\R^n} K_1(x,x_1) K_2(x_1,x_2)\dots K_n(x_{n-1},x_n)f(x_n) dx_1dx_2\dots dx_n, 
\end{equation}
has kernel 
\begin{equation}\label{kernelTn}
K(x,x_n)=\int_{\R^{n-1}} K_1(x,x_1)K_2(x_1,x_2)\dots K_n(x_{n-1},x_n) dx_1\dots dx_{n-1},
\end{equation}
and trace
\begin{equation}\label{traceTn}
\text{tr}(T_1T_2\dots T_n)=\int_{\R^n} K_1(x,x_1)K_2(x_1,x_2)\dots K_n(x_{n-1},x) dx dx_1\dots dx_{n-1}.
\end{equation}
We also have
\begin{equation}\label{prodtraceTn}
|\text{tr}(T_1T_2\dots T_n)|\leq \prod_{j=1}^{n} \|T_j\|.
\end{equation}
Let $m(\xi)$ a real function, we defined the multiplier operator $M$ associated to $m$, as
$$
\widehat{M f}(\xi)=m(\xi)\widehat{f}(\xi).
$$
Let $u\in S(\R)$ and $M$ a multiplier operator associated to $m$,  we define other operators $Mu$ and $uM$ as follow $(Mu)f(x):=M(uf)(x)$, and $(uM)f(x)=u(x)(Mf)(x)$. 

If $M_j$ is a multiplier operator with multiplicator $m_j$, $j=1,2$, since that $M_1M_2=M_2M_1$ has multiplicator $m_1m_2=m_2m_1$, then 
\begin{equation}\label{ConmOper}
M_1M_2u=M_2M_1u\quad \textrm{and} \quad uM_1M_2=uM_2M_1.
\end{equation}
We have the following examples
\\
\\
\begin{tabular}{|c|c|c|}
\hline
Operador & Kernel,  $K(x,\xi)=$ & Trace \\
\hline
M & $m^{\vee}(x-\xi)$ &  \\
\hline
Mu & $m^{\vee}
(x-\xi) u(\xi)$ & $\int m^{\vee}
(0) u(x)dx=\left(\int m(x)dx \right) \left(\int u(x)dx\right)$ \\
\hline
uM & $u(x) m^{\vee}
(x-\xi) $  &$\int u(x) m^{\vee}
(0) dx=\left(\int m(x)dx\right) \left(\int u(x)dx\right)$ \\
\hline
uMv & $u(x) m^{\vee}
(x-\xi) v(\xi)$ &$\int u(x) m^{\vee}
(0)  v(x)dx=\left(\int m(x)dx \right) \left(\int u(x) v(x) dx\right)$\\
\hline
\end{tabular}
\\
\\
\\
Observe that using \eqref{ProdTn} and \eqref{kernelTn} we can define other operators as $M_1uM_2=M_1(uM_2)=(M_1u)M_2$ with kernel:
\begin{equation}\label{eqKxav}
K(x,z):=K(m_1,u,m_2)=\int_{\R} K_1(x,y) K_2(y,z) dy=\int_{\R} m_1^{\vee}(x-y) u(y) m_2^{\vee}
(y-z) dy,
\end{equation}
and using \eqref{traceTn}, with trace:
\begin{equation}
 \quad \text{tr}(M_1uM_2)=\int_{\R^2} K_1(x,y) K_2(y,x) dy dx=\int_{\R^2} m_1^{\vee}(x-y) u(y) m_2^{\vee}
(y-x) dydx.
\end{equation}
It is not difficult to see that if $m_1, m_2 \in L^2(\R)$, then $K(x,z)\in L^2(\R^2)$. 

In order to prove \eqref{Pedrx1}, we need the following elemental lemma
 \begin{lemma}\label{lemabasic}
 If $a,b>0$ and $a+b>1$, we have
\begin{equation}\label{lfc1-1}
\int_{\R} \dfrac{dx}{\langle x -\alpha \rangle^{a}\langle x -\beta \rangle^{b}} \lesssim \dfrac{1}{\langle \alpha -\beta \rangle^{c}},  \quad c=\min\{a,b, a+b-1\}.
\end{equation}
\end{lemma}

\begin{lemma}\label{lemaPedr}
For $k\neq 0$ and $u\in S(\R)$,
\begin{equation}\label{Pedrx1}
\|(k-\partial)^{-1/2}u(k+\partial)^{-1/2}\|^2\lesssim \int_{\R}\dfrac{|\widehat{u}(\xi)|^2}{|k|+|\xi|}d\xi\leq C_k\|u\|^2_{H^{-1/2}},
\end{equation}
where $C_k=(\min\{1,|k|\})^{-1/2}$.
\end{lemma}
 
\begin{proof}
From \eqref{eqKxav}, using Plancherel's identity and properties of the Fourier transform we have
\begin{equation}
\begin{split}
\|K(x,z)\|_{L^2_x L^2_z}&=\|m_1^{\vee}*[u(\cdot) (m_2(\cdot)e^{-iz(\cdot)})^{\vee}](x)\|_{L^2_x L^2_z}\\
&=\|m_1(\xi) \|[u(\cdot) (m_2(\cdot)e^{-iz(\cdot)})^{\vee}]^{\wedge}(\xi)\|_{L^2_z} \|_{ L^2_\xi}\\
&=\|m_1(\xi) \|[\widehat{u}* (m_2(\cdot)e^{-iz(\cdot)})(\xi)\|_{L^2_z} \|_{ L^2_\xi}\\
&=\|m_1(\xi) \|[(u(\cdot) e^{-i\xi(\cdot)})^{\vee} m_2(\cdot)]^{\wedge}(z)\|_{L^2_z} \|_{ L^2_\xi}\\
&=\|m_1(\xi) \|u^{\vee} (\xi-\eta) m_2(\eta)\|_{L^2_\eta} \|_{ L^2_x}.
\end{split}
\end{equation}
Thus making a change of variables $\xi =k \xi$ and $\eta=k\eta$ followed by another change of variables $\xi-\eta=y$,
\begin{equation}
\begin{split}
\|K(x,z)\|_{L^2_x L^2_z}^2&=\int_{\R^2}\dfrac{ |\widehat{u} (\xi-\eta)|^2}{\sqrt{k^2+\xi^2} \sqrt{k^2+\eta^2}} d\xi d\eta\\&\sim\int_{\R^2}\dfrac{ |\widehat{u} (k(\xi-\eta))|^2}{\langle \xi \rangle \langle \eta \rangle} d\xi d\eta\\
&=\int_{\R^2}\dfrac{ |\widehat{u} (k y)|^2}{\langle y+\eta \rangle \langle \eta \rangle} dy d\eta\\
&\lesssim\int_{\R}\dfrac{ |\widehat{u} (k y)|^2}{\langle y \rangle } dy,
\end{split}
\end{equation}
where in the last inequality was used Fubinni theorem and Lemma \ref{lemabasic} with $a=b=1$.
\end{proof}

If $T$ is a linear operator associated with the kernel $K$, we define the operator $\overline{T}$ as
\begin{equation}
\overline{T}f(x):=\int_{\R} \overline{K(x,y)} f(y) dy,
\end{equation}
where $\overline{K}$ denotes the complex conjugate of $K$. Thus if $K(x,y) \in \R$, then $\overline{T}=T$. 
Observe that $\overline{m^{\vee}(\eta)}=\widehat{\overline{m}}(\eta)=\overline{m}^{\vee}(-\eta)$, thus
\begin{equation}\label{conjM}
\overline{Mu}=M^{-} \overline{u},
\end{equation}
where $M^{-}$ is the multiplier operator, associated to $m^{-}(\xi)=\overline{m}(-\xi)$. In this way we get, for any $k$ real number
\begin{equation}
\overline{(k-\partial)^{-1}u }=(k-\partial)^{-1}\overline{u}, \quad \textrm{and}\quad \overline{(k+\partial)^{-1}u }=(k+\partial)^{-1}\overline{u},
\end{equation}
since that if $m(\xi)=(k\pm i\xi)^{-1}$, then $\overline{m}(-\xi)=m(\xi)$.

Let $M$ a multiplier operator associated to $m$,  by \eqref{kerAdj} 
$$K_{M^*}(x,y)=\overline{K_{M}(y,x)}=\overline{m^{\vee}}(y-x)=\overline{\widehat{m}}(x-y)=\overline{m}^{\vee}(x-y)=K_{\overline{M}}(x,y), $$ i.e. 
\begin{equation}\label{adj1}
M^{*}=\overline{M}.
\end{equation}
 Similarly using \eqref{kerAdj} and the above example, we have
\begin{equation}\label{adj2}
(Mu)^{*}=\overline{u} \overline{M}. 
\end{equation}
The equalities \eqref{adj1} and \eqref{adj2} imply that
\begin{equation}\label{adj3}
(M_1uM_2)^{*}=((M_1u)M_2)^{*}=M_2^{*}(M_1u)^{*}=\overline{M_2}\overline{u} \overline{M_1}.
\end{equation}

\subsection{Derivatives of the Multiplication Operator}
Let $u\in S(\R)$ we define $Pu: S(\R) \to S(\R)$ the operator of multiplication associated to $u$ as:
$$
(Pu)f=uf, \quad f\in S(\R).
$$
We will use the notation $Pu:=u$.  Let $n,l \in \Z^+$, we also define the operators of multiplication $u \partial^{n}: S(\R) \to S(\R)$ by $(u \partial^n)(f)=u f^{(n)}$ and $(\partial^l u): S(\R) \to S(\R)$ by $(\partial^l u)(f):=(uf)^{(l)}$.
Considering $n=l=1$, we obtain $(\partial u)(f)=(uf)'=u'f+uf'$ for all $f\in S(\R)$  or equivalently $\partial u =u' +u \partial$ and thus $u'=\partial u-u \partial=[\partial, u]$.
We can have combinations between both operators such as 
\begin{equation}\label{eq2sc1}
\partial(u\partial)(f)=((u\partial)f)'=(uf')'=u'f'+uf'' \quad \iff\quad \partial(u\partial)=u'\partial+u\partial^2.
\end{equation}

Similarly we can define other operators of multiplication, such as
\begin{equation}
(\partial^2 u)(f):=(uf)''=u''f +2u'f'+uf''
\end{equation}
which gives
\begin{equation}
\partial^2 u=u''+2u'\partial +u \partial^2 \quad \iff\quad u''=\partial^2 u-2u'\partial -u \partial^2,
\end{equation}
We define $(k\pm\partial)u:=ku\pm\partial u$, $(k\pm\partial)^2u:=k^2u\pm2k\partial u+\partial^2 u$ and by induction $(k\pm\partial)^nu:=(k\pm\partial)^{n-1}(k\pm\partial)u$, $n=3,4,\dots$.

Using  \eqref{eq2sc1}, adding and subtracting terms in the above equality, we have the following identity
\begin{equation}\label{eq2sc2}
\begin{split}
u''=&u(\partial^2-2k\partial-k^2)+(\partial^2+ 2k\partial-k^2)u+2(k-\partial)u(k+\partial)\\
=&u(k-\partial)^2+(k+\partial)^2u-4k^2u+ 2(k-\partial)u(k+\partial).
\end{split}
\end{equation}
valid for all $k\in \R$.  Also we have  the operator of multiplication
\begin{equation}
(\partial^3 u)(f):=(uf)'''=u'''f +3u''f'+3u'f''+uf''''\quad 
\end{equation}
therefore
\begin{equation}
\partial^3 u=u''' +3u'' \partial +3u'\partial^2+u\partial^3\,\,
\iff\,\,
u'''=\partial^3 u-3u'' \partial-3u'\partial^2-u\partial^3
\end{equation}
we also have the following identity
$$
u'''=\partial^3 u-3\partial^2 u \partial+3\partial u \partial^2-u \partial^3=\sum_{j=0}^3(-1)^j\binom{3}{j}\partial^{3-j} u \partial^j.
$$
 Using  $u'=\partial u-u \partial$, adding and subtracting terms in the above equality, we obtain
\begin{equation}\label{u3}
u'''=u(k-\partial)^3+(k+\partial)^3u-8k^3u+(k-\partial)(3u'+6k u)(k+\partial)
\end{equation}
for any $k\in \R$. On the other hand
$
\partial(|u|^2u)(f)=(|u|^2u f)'=2|u|^2u'f+u^2\overline{u}'f+|u|^2u f'
$
thus
\begin{equation}\label{nonl1}
\partial(|u|^2u)=2|u|^2u'+u^2\overline{u}'+|u|^2u \partial,
\end{equation}
adding and subtracting terms in the above equality, we obtain
\begin{equation}\label{nonl2}
2|u|^2u'=-(|u|^2u) (k+\partial)- (k-\partial)(|u|^2u)-u^2(\overline{u}'-2k\overline{u}),
\end{equation}
\\
\subsection{Trace of Products of Multiplier Operators}
Next, we will state some properties of the Products of Multiplier Operators.
\begin{proposition}
Let $M_j$ the multipliers operator  associated to $m_j$, $u_j \in S(\R)$, $j=1,\dots,n+1$ and $\sigma$ a shift permutation of the n-upla $(1,2, \dots, n)$,  we have
\begin{equation}\label{Prop1Tn}
\text{tr}\left( \Pi_{j=1}^n  M_j u_j\right)=\text{tr}\left( \Pi_{j=1}^n  M_{\sigma(j)} u_{\sigma(j)}\right)=\text{tr}\left( \left(\Pi_{j=1}^{n-1} u_{\sigma(j)} M_{\sigma(j+1)}\right) u_{\sigma(n)} M_{\sigma(1)} \right)
\end{equation}
\begin{equation}\label{Prop2Tn}
\text{tr}\left( \left(\Pi_{j=1}^{n}  M_ju_j\right) M_{n+1} \right)=\text{tr}\left( M_1M_{n+1} u_1\left(\Pi_{j=2}^{n} M_{j}  u_{j} \right)\right)
\end{equation}
\begin{equation}\label{Prop3Tn}
\text{tr}\left( \Pi_{j=1}^n u_j M_j \right)=\text{tr}\left( \left( \Pi_{j=1}^{n-1} M_j u_{j+1} \right)M_n u_1\right)
\end{equation}
\begin{equation}\label{Prop4Tn}
\text{tr}\left( \left( \Pi_{j=1}^n u_j M_j\right) u_{n+1} \right)=\text{tr}\left( \left( \Pi_{j=1}^{n-1} M_j u_{j+1}\right) M_n (u_1 u_{n+1})\right)
\end{equation}
\end{proposition}
\begin{proof}
Using \eqref{traceTn} and example before, we have
\begin{equation}
\text{tr}\left( \Pi_{j=1}^n  M_j u_j\right)=\int_{\R^n}m_1^{\vee}
(x-\xi_1) u_1(\xi_1)m_2^{\vee}
(\xi_1-\xi_2) u_2(\xi_2)\cdots m_n^{\vee}
(\xi_{n-1}-x) u_n(x)dxd\xi_1 \cdots d \xi_{n-1}
\end{equation}
applying Fubinni we get \eqref{Prop1Tn}.
Similarly
\begin{equation*}
\begin{split}
&\text{tr}\left( \left(\Pi_{j=1}^{n}  M_ju_j\right) M_{n+1} \right)= \\
&\int_{\R^{n+1}}m_1^{\vee}
(x-\xi_1) u_1(\xi_1)m_2^{\vee}
(\xi_1-\xi_2) u_2(\xi_2)\cdots m_{n}^{\vee}
(\xi_{n-1}-\xi_{n}) u_{n}(\xi_{n})  m_{n+1}^{\vee}
(\xi_{n}-x)\\
=&\int_{\R^{n}}u_1(\xi_1)m_2^{\vee}
(\xi_1-\xi_2) u_2(\xi_2)\cdots m_{n}^{\vee}
(\xi_{n-1}-\xi_{n}) u_{n}(\xi_{n})  \int_{\R}m_1^{\vee}
(x-\xi_1) m_{n+1}^{\vee}
(\xi_{n}-x) dx\\
=&\int_{\R^{n}}u_1(\xi_1)m_2^{\vee}
(\xi_1-\xi_2) u_2(\xi_2)\cdots m_{n}^{\vee}
(\xi_{n-1}-\xi_{n}) u_{n}(\xi_{n})  (m_1m_{n+1})^{\vee}
(\xi_{n}-\xi_1) dx
\end{split}
\end{equation*}
and \eqref{Prop2Tn} follows. The proof of inequalities \eqref{Prop3Tn} and \eqref{Prop4Tn} are similar.
\end{proof}

Note that on the right side of \eqref{Prop2Tn}, \eqref{Prop3Tn} and \eqref{Prop4Tn} we can  still use the \eqref{Prop1Tn} property.  By property \eqref{Prop3Tn} we have
\begin{equation}\label{trace2}
\begin{split}
\text{tr}\left(M_1 u_1 M_2u_2\right)=&\text{tr}\left(u_2 M_1 u_1M_2\right)=\int_{\R^{2}}m_1^{\vee}
(x-\xi_1) u_1(\xi_1)m_2^{\vee}
(\xi_1-x) u_2(x) d\xi_1 dx\\
=&\int_{\R}u_1(\xi_1) \left\{ \left( \widehat{m_1}m_2^{\vee} \right)*u_2 \right\} (\xi_1) d\xi_1.
\end{split}
\end{equation}
\begin{lemma}\label{lemaPedr2}
For $k\neq0$ and $u\in S(\R)$,
\begin{equation}\label{Pedrx2}
\text{Re\,tr}\left\{(k-\partial)^{-1}u(k+\partial)^{-1}\overline{u}\right\}= 2kc\int_{\R}\dfrac{|\widehat{u}(\xi)|^2}{4 k^2+\xi^2}d\xi\sim_k\|u\|^2_{H^{-1}}.
\end{equation}
\end{lemma}
\begin{proof}
Using \eqref{ConmTr}, \eqref{trace2} with $u_1=\overline{u}$, $u_2=u$ and $m_1(\xi)=\dfrac{1}{k+i\xi}$, $m_2(\xi)=\dfrac{1}{k-i\xi}=\overline{m_1}$ and Plancherel identity, we have $m_2^{\vee}=\overline{m_1}^{\vee}=\overline{\widehat{m_1}}$
\begin{equation}\label{eq2Pedrx2}
\begin{split}
\text{Re\,tr}\left\{(k-\partial)^{-1}u(k+\partial)^{-1}\overline{u}\right\}=&\text{Re\,tr}\left\{(k+\partial)^{-1}\overline{u}(k-\partial)^{-1}u\right\}\\
=& \text{Re\,}\int_{\R}\overline{\widehat{u}}(\xi)\widehat{u}(\xi) \widehat{\left( \widehat{m_1}m_2^{\vee} \right)}(\xi) d\xi\\
=&\text{Re\,}\int_{\R}|\widehat{u}(\xi)|^2 \widehat{|\widehat{m_1}|^2}(\xi) d\xi\\
= &c\text{Re\,}\int_{\R}|\widehat{u}(\xi)|^2 \dfrac{1}{2k-i\xi} d\xi.
\end{split}
\end{equation}
\end{proof}

From the definition of norm  of operator $M_1uM_2$, (see definition  \eqref{definNormaoper}), \eqref{adj3} and property \eqref{Prop2Tn} holds
\begin{equation}
\|M_1uM_2\|^2=\text{tr}(M_1uM_2\overline{M_2}\overline{u} \overline{M_1})=\text{tr}(\overline{M_1}M_1uM_2\overline{M_2}\overline{u} ),
\end{equation}
similarly
\begin{equation}
\|M_2\overline{u}M_1\|^2=\text{tr}(M_2\overline{u}M_1\overline{M_1}u \overline{M_2})=\text{tr}(\overline{M_2}M_2\overline{u} M_1\overline{M_1}u)
\end{equation}
the above equalities,  \eqref{ConmTr} and \eqref{ConmOper} imply that 
\begin{equation}\label{ignorm}
\|M_1uM_2\|=\|M_2\overline{u}M_1\|.
\end{equation}

By analogy with has gone in the case of the NLS and mKdV, we define
\begin{equation}\label{definalpha}
\alpha(u(t),k):=\text{Re} \sum_{l=1}^{\infty} \dfrac{(-1)^{l-1}}{l} \text{tr} \left\{ [(k-\partial)^{-1/2}u(k+\partial)^{-1}\overline{u} (k-\partial)^{-1/2}]^l\right\}
\end{equation}
Let $R^{\pm}=(k\pm\partial)^{-1}$. Using \eqref{Prop2Tn} we deduce
\begin{equation}\label{eq1}
\text{tr} \left\{ [(k-\partial)^{-1/2}u(k+\partial)^{-1}\overline{u} (k-\partial)^{-1/2}]^l\right\}=\text{tr} \left\{   \left(R^{-}uR^{+}\overline{u}\right)^l\right\}.
\end{equation}
We consider $u$ a solution of h-NSE, i.e.
\begin{equation}\label{modeloHONSE}
\partial_{t}u=-L(u)+F(u),
 \quad x,t \in \R, 
\end{equation}
where $L(u)=ia \partial^{2}_{x}u+b\partial^{3}_{x}u$ and $F(u)=2ia|u|^{2}u+6b
|u|^{2}\partial_{x}u$.
\begin{theorem}\label{alphaHonse}[Conservation of $\alpha$ for h-NSE]
Let $u(x,t)$ denote a Schwartz-space solution to HONSE. Then for any $k> 0$ holds,
\begin{equation}\label{eqalpha}
\dfrac{d}{dt}\alpha(u(t),k)=0.
\end{equation}
\end{theorem}
\begin{proof}
Using \eqref{eq1} and differentiating
\begin{equation}
\begin{split}
\dfrac{d}{dt}\alpha(u(t),k)&=\text{Re} \sum_{l=1}^{\infty} (-1)^{l-1}\text{tr} \left\{ \left(R^{-}uR^{+}\overline{u}\right)^{l-1}\left(  R^{-}(\dfrac{d}{dt}u)R^{+}\overline{u} +R^{-}uR^{+}\dfrac{d}{dt}\overline{u}\right)\right\}\\
&=\text{Re} \sum_{l=1}^{\infty} (-1)^{l-1}\text{tr} \left\{ \left(R^{-}uR^{+}\overline{u}\right)^{l-1}\left(  R^{-}(-Lu)R^{+}\overline{u} +R^{-}uR^{+}(-\overline{Lu})\right)\right\}\\
&+\text{Re} \sum_{l=1}^{\infty} (-1)^{l-1}\text{tr} \left\{ \left(R^{-}uR^{+}\overline{u}\right)^{l-1}\left(  R^{-}(F(u))R^{+}\overline{u} +R^{-}uR^{+}(\overline{F(u)}\right)\right\}\\
&=\text{Re}\,\text{tr} \left\{ \left(  R^{-}(-Lu)R^{+}\overline{u} +R^{-}uR^{+}(-\overline{Lu})\right)\right\}\qquad (l=1)\\
&+\text{Re} \sum_{l=1}^{\infty} (-1)^{l}\text{tr} \left\{ \left(R^{-}uR^{+}\overline{u}\right)^{l}\left(  R^{-}(-Lu)R^{+}\overline{u} +R^{-}uR^{+}(-\overline{Lu})\right)\right\}\quad (l:=l+1)\\
&+\text{Re} \sum_{l=1}^{\infty} (-1)^{l-1}\text{tr} \left\{ \left(R^{-}uR^{+}\overline{u}\right)^{l-1}\left(  R^{-}(F(u))R^{+}\overline{u} +R^{-}uR^{+}(\overline{F(u)}\right)\right\}.
\end{split}
\end{equation}
Thus \eqref{eqalpha} follows if 
\begin{equation}\label{Lin1}
\text{Re}\,\text{tr} \left\{ \left(  R^{-}(Lu)R^{+}\overline{u} +R^{-}uR^{+}(\overline{Lu})\right)\right\} =0,
\end{equation}
and
\begin{equation}\label{Fnlin1}
\begin{split}
\text{Re}\,\text{tr} \left\{ \left(R^{-}uR^{+}\overline{u}\right)^{l}\left(  R^{-}(Lu)R^{+}\overline{u} +R^{-}uR^{+}(\overline{Lu})\right)\right\}=&\\
-\text{Re}\, \text{tr} \left\{ \left(R^{-}uR^{+}\overline{u}\right)^{l-1}\left(  R^{-}(F(u))R^{+}\overline{u} +R^{-}uR^{+}(\overline{F(u)}\right)\right\}.
\end{split}
\end{equation}
Using \eqref{trace2} with $M_1=R^{-}$ and $M_2=R^{+}$, integrating three times we get
\begin{equation}\label{x1}
\text{tr} \left\{ \left(  R^{-} (u_{xxx})R^{+}\overline{u} \right)\right\}=\int_{\R}u_{xxx}(\xi_1) \left( \widehat{m_1}m_2^{\vee} \right)*\overline{u}  (\xi_1) d\xi_1=-\int_{\R}u(\xi_1) \left( \widehat{m_1}m_2^{\vee} \right)*\overline{u_{xxx}}  (\xi_1) d\xi_1,
\end{equation}
similarly
\begin{equation}\label{x2}
\text{tr} \left\{ \left(  R^{-} (iu_{xx})R^{+}\overline{u} \right)\right\}=\int_{\R}iu_{xx}(\xi_1) \left( \widehat{m_1}m_2^{\vee} \right)*\overline{u}  (\xi_1) d\xi_1=-\int_{\R}u(\xi_1) \left( \widehat{m_1}m_2^{\vee} \right)*\overline{iu_{xx}}  (\xi_1) d\xi_1,
\end{equation}
using \eqref{x1} and \eqref{x2} we obtain \eqref{Lin1}. In order to prove \eqref{Fnlin1},
considering k:=-k and taking complex conjugate in \eqref{eq2sc2}, \eqref{u3} we obtain
\begin{equation}\label{eq2sc2conj}
\overline{u_{xx}}=(k-\partial)^2\overline{u}+\overline{u}(k+\partial)^2-4k^2\overline{u}+2(k+\partial)\overline{u}(k-\partial).
\end{equation}
and
\begin{equation}\label{u3-1}
\begin{split}
\overline{u_{xxx}}=&\overline{u}(-k-\partial)^3+(-k+\partial)^3\overline{u}+8k^3\overline{u}+(-k-\partial)(3\overline{u_x}-6k \overline{u})(-k+\partial)\\
=&-\overline{u}(k+\partial)^3-(k-\partial)^3\overline{u}+8k^3\overline{u}+(k+\partial)(3\overline{u_x}-6k \overline{u})(k-\partial).
\end{split}
\end{equation}
Using, \eqref{u3}, \eqref{u3-1}, \eqref{eq2sc2} and  \eqref{eq2sc2conj} it is not dificult to see that 
\begin{equation}\label{der3u3-1}
\begin{split}
 R^{-}(u_{xxx})R^{+}\overline{u} +R^{-}uR^{+}(\overline{u_{xxx}})=& R^{-}(k+\partial)^3uR^{+}\overline{u}+3u_x\overline{u}+6k|u|^2\\
 & -R^{-}uR^{+}\overline{u}(k+\partial)^3+R^{-}(3u\overline{u_x}-6k |u|^2) (k-\partial),
\end{split}
\end{equation}
and 
\begin{equation}\label{der2u3-1}
\begin{split}
 R^{-}(iu_{xx})R^{+}\overline{u} +R^{-}uR^{+}(\overline{iu_{xx}})=i\left(R^{-}\mathcal{A}uR^{+}\overline{u}+2|u|^2 -R^{-}uR^{+}\overline{u}\mathcal{A}-2R^{-}(|u|^2) (k-\partial)\right),
\end{split}
\end{equation}
where $\mathcal{A}=(k+\partial)^2$.

From \eqref{der3u3-1} and \eqref{der2u3-1} we have that
\begin{equation}\label{xder2u3-1}
R^{-}L(u)R^{+}\overline{u} +R^{-}uR^{+}\overline{L(u)}=R^{-}\mathcal{B} uR^{+}\overline{u} - R^{-}uR^{+}\overline{u}\mathcal{B}+F_1(k)+R^{-}\overline{F_1(-k) }(k-\partial),
\end{equation}
where $F_1(k)=2ai|u|^2+ 3bu_x\overline{u}+6bk|u|^2$ and $\mathcal{B}= ia\mathcal{A}+b(k+\partial)^3$, using  Property \eqref{Prop1Tn} we have
\begin{equation}\label{eqB1}
\text{Re}\,\text{tr} \!\left\{ \left(R^{-}uR^{+}\overline{u}\right)^{l}\left(  R^{-}\mathcal{B} uR^{+}\overline{u}\right)\right\}=\text{Re}\,\text{tr} \!\left\{ \left(  R^{-}\mathcal{B}uR^{+}\overline{u}\right)\left(R^{-}uR^{+}\overline{u}\right)^{l}\right\},
\end{equation}
and using Property \eqref{Prop2Tn} we get
\begin{equation}\label{eqB2}
\text{Re}\,\text{tr} \!\left\{ \left(R^{-}uR^{+}\overline{u}\right)^{l}\left(  R^{-}uR^{+}\overline{u}\mathcal{B}\right)\right\}=\text{Re}\,\text{tr} \!\left\{ \left(  R^{-}\mathcal{B}uR^{+}\overline{u}\right)\left(R^{-}uR^{+}\overline{u}\right)^{l}\right\}
\end{equation}
consequently from \eqref{xder2u3-1}, \eqref{eqB1} and \eqref{eqB2} it follows that
\begin{equation}\label{Fnlin3}
\begin{split}
\text{Re}\,&\text{tr} \!\left\{ \left(R^{-}uR^{+}\overline{u}\right)^{l}\left(  R^{-}L(u)R^{+}\overline{u} +R^{-}uR^{+}\overline{L(u)}\right)\right\}
\!=\!\text{Re}\,\text{tr}\! \left\{\left( R^{-}uR^{+}\overline{u}\right)^{l-1} R^{-}uR^{+} \overline{u}F_1(k)\right\}
\\
&+\text{Re}\,\text{tr} \left\{ R^{-}uR^{+}\overline{u}\left( R^{-}uR^{+}\overline{u}\right)^{l-1} R^{-}\overline{F_1(-k) }(k-\partial)\right\}\\
=&
\text{Re}\,\text{tr}\! \left\{\left( R^{-}uR^{+}\overline{u}\right)^{l-1} R^{-}uR^{+} (\overline{u}F_1(k))\right\}
+\text{Re}\,\text{tr} \left\{ \left(R^{-}u R^{+}\overline{u}\right)^{l-1} R^{-}u\left(\overline{F_1(-k) } \right)R^{+}\overline{u}\right\}
\end{split}
\end{equation}
 in the last equality was used  \eqref{Prop1Tn} (with the permutation $\sigma(n)=2$, $\sigma(n-1)=1$ and $\sigma(j)=j+2, j=1, \dots n-2$, similar as in \eqref{eqB1}). 
Now, we will to consider the term with $|u|^2u_x$ of $F(u)$, considering k:=-k and taking complex conjugate in \eqref{nonl2}, we obtain
\begin{equation}\label{nonl2conj}
\begin{split}
2|u|^2\overline{u_x}=&-(|u|^2\overline{u}) (-k+\partial)- (-k-\partial)(|u|^2\overline{u})-\overline{u}^2(u_x+2ku)\\=&(k+\partial)|u|^2\overline{u}+(|u|^2\overline{u})(k-\partial)-\overline{u}^2(u_x+2ku),
\end{split}
\end{equation}


By \eqref{nonl2} and \eqref{nonl2conj} we get
\begin{equation}\label{xNonlx1}
\begin{split}
 R^{-}(2|u|^2u_x)R^{+}\overline{u} &+R^{-}uR^{+}(2|u|^2\overline{u_x})= R^{-}(-|u|^4)-|u|^2u R^{+}\overline{u}+ R^{-}(-u^2\overline{u_x}+2k|u|^2 u)R^{+}\overline{u}\\
&+ R^{-}(|u|^4) +R^{-}u R^{+}(|u|^2\overline{u})(k-\partial)+R^{-}uR^{+}(-\overline{u}^2u_x-2k|u|^2\overline{u})
\end{split}
\end{equation}

and using the properties \eqref{Prop1Tn} we have
\begin{equation}\label{Nonlx1}
\begin{split}
\text{Re}&\, \text{tr} \left\{ \left(R^{-}uR^{+}\overline{u}\right)^{l-1}\left(  R^{-}(2|u|^2u_x)R^{+}\overline{u} +R^{-}uR^{+}(2|u|^2\overline{u_x}\right)\right\}=\\
&\text{Re}\, \text{tr} \left\{ \left(R^{-}uR^{+}\overline{u}\right)^{l-1}\left(  R^{-}(-u^2\overline{u_x}+2k|u|^2 u)R^{+}\overline{u} +R^{-}uR^{+}(-\overline{u}^2u_x-2k|u|^2\overline{u}\right)\right\}
\end{split}
\end{equation}
where the second term in RHS\eqref{xNonlx1} in the first line is canceled with the second term in RHS\eqref{xNonlx1} in the second line (using \eqref{Prop1Tn}, similar as in \eqref{eqB1}.
Therefore
\begin{equation}\label{Nonlx3}
\begin{split}
&\text{Re}\, \text{tr} \left\{ \left(R^{-}uR^{+}\overline{u}\right)^{l-1}\left(  R^{-}(F(u))R^{+}\overline{u} +R^{-}uR^{+}(\overline{F(u)}\right)\right\}=\\
&\text{Re}\, \text{tr} \left\{ \left(R^{-}uR^{+}\overline{u}\right)^{l-1}\left(  R^{-}( - u\left(\overline{F_1(-k) } \right)  )R^{+}\overline{u} +R^{-}uR^{+}(-\overline{u}F_1(k)\right)\right\},
\end{split}
\end{equation}
where $\overline{u}F_1(k)=2ia|u|^2\overline{u}+3b\overline{u}^2u_x+6k b|u|^2\overline{u}$. This equality together with \eqref{Fnlin3} prove equality \eqref{Fnlin1} and therefore proves the theorem

\end{proof}
\begin{proposition}\label{Rem1} Let $u$ solution of the IVP  \eqref{honse} and $v$ such that
\begin{equation}\label{gauge}
u(x,t) := v(x+d_1t, t)e^{i(d_2x+d_3t)}.
\end{equation}
where $d_j\in \R$, $j=1,2,3$. If $\dfrac{d\alpha(u(t),k)}{dt}=0$, then $\dfrac{d\alpha(v(t),k)}{dt}=0$.
\end{proposition}
\begin{proof}
By the definition of $\alpha$ and \eqref{eq1}, holds
\begin{equation}
\alpha(v(t),k):=\text{Re} \sum_{l=1}^{\infty} \dfrac{(-1)^{l-1}}{l} \text{tr} \left\{   \left(R^{-}ue^{-id_2x}R^{+}\overline{u}e^{id_2x}\right)^l\right\}.
\end{equation}
and using Property \eqref{Prop1Tn} we have
\begin{equation}
\alpha(u(t),k):=\text{Re} \sum_{l=1}^{\infty} \dfrac{(-1)^{l-1}}{l} \text{tr} \left\{   \left(\mathcal{R}^{-}u\mathcal{R}^{+}\overline{u}\right)^l\right\}.
\end{equation}
where $\mathcal{R}^{\pm}=e^{\mp id_2x} R^{\pm}$. The rest of the proof is equal to the proof of Theorem \ref{alphaHonse}, replacing $R^{-}$ by $\mathcal{R}^{-}$ and $R^{+}$ by $\mathcal{R}^{+}$.
\end{proof}
\begin{remark}
 We consider the following Gauge transform 
\begin{equation}\label{gauge}
u(x,t) := v(x+d_1t, t)e^{i(d_2x+d_3t)}.
\end{equation}

Using this transformation the IVP \eqref{honse} turns out to be
\begin{equation}\label{xhonse1} 
\!\!\!\!\!\!\!\begin{cases}
\partial_{t}v+ b\partial^{3}_{x}v+i(a +3bd_2)
\partial^{2}_{x}v+i(d_3-bd_2^3-a d_2^2)v+(d_1-3bd_2^2-2a d_2)\partial_x v  \\
+ i(c_1+c_2d_2) |v|^2v  +c_2|v|^2v_x=  0,\\
v(x,0) = v_0(x) := u_0(x)e^{-id_2x}.
\end{cases}
\end{equation} 
   
 If one chooses $d_1=-\frac{a^2}{3b}$, $d_2= \frac{-a}{3b}$, $d_3=\frac{2a^3}{27b^2}$, the third, fourth, fifth terms in the first equation in \eqref{xhonse1} vanish i.e.
 \begin{equation}\label{e-nls1} 
\!\!\!\!\!\!\!\begin{cases}
\partial_{t}v+ b\partial^{3}_{x}v + ic_3 |v|^2v+c_2 |v|^2v_x=  0,\\
v(x,0) = v_0(x) := u_0(x)e^{iax/(3b)},
\end{cases}
\end{equation} 
where $c_3=c_1-\frac{ac_2}{3b}$.
 Also, we note that
 $$\|u_0\|_{H^s}\sim \|v_0\|_{H^s}\quad \mathrm{and}\quad \|u\|_{Z^{s,b}} \sim \|v\|_{X^{s,b}},
 $$
 where $X^{s,b}$ is the Fourier transform norm space with phase function $\xi^3$, i.e., with norm
 $$\|u\|_{X^{s,b}}:=\|\langle\xi\rangle^s\langle\tau-\xi^3\rangle^b\widehat{u}(\xi, \tau)\|_{L^2_{\xi}L^2_{\tau}}.$$
This observation allows one to consider the parameter $a =0$.     
\end{remark}

\section{Apriori estimates}\label{sec-5}

\begin{proposition}\label{gwpapr}
Let $2\leq p<\infty$ and $u\in S(\R)$ a solution to \eqref{xhonse}. Then, there exists $C=C(p)$ positive such that
\begin{equation}
\|u(t)\|_{M^{2,p}(\R)}\leq C\left( 1+\|u(0)\|_{M^{2,p}(\R)}  \right)^{p/2-1}\|u(0)\|_{M^{2,p}(\R)},
\end{equation}
for any $t\in \R$.
\end{proposition}
\begin{proof}
The proof is done in two steps.
Initially we will to consider the initial data 
\begin{equation}\label{dadoi}
\|u_0\|_{M^{2,p}}<\epsilon\ll1.
\end{equation}

Let $u$ solution of equation \eqref{modeloHONSE}, for any $n\in \mathbb{Z}$ we define $u_n$ such that
\begin{equation}\label{tta}
u_n(x, t) = e^{- i n x}e^{i (a n^2 + 2 b n^3) t}u(x - (2 a n + 3 b n^2)t, t)
\end{equation}
then $u_n$ is a solution of
\begin{equation}\label{honse3}
\begin{cases} 
\partial_{t}v+iA 
\partial^{2}_{x}v+ b\partial^{3}_{x}v=2iA |v|^{2}v+6b|v|^{2}\partial_{x}v,\quad x,t \in \R,\\
 v(x,0) = e^{-inx}u(x,0),
 \end{cases}
\end{equation}
where $A= a+3bn$, and also have $\widehat{u}(\xi,t)=e^{i(an^2+bn^3)t}e^{i(2an+3bn^2)(\xi-n)t}\widehat{u_n}(\xi-n,t)$, thus
\begin{equation}\label{ign}
|\widehat{u_n}(\xi,t)|=|\widehat{u}(\xi+n,t)|.
\end{equation}
Theorem \ref{alphaHonse} and Remark \ref{Rem1}  imply that
\begin{equation}\label{dze}
\dfrac{d\alpha (u_n(t))}{dt}=0,
\end{equation}

for any $n\in \mathbb{Z}$ and $\xi,t \in \R$.

We denote  $\alpha(u_{n}(t)):= \alpha( u_{n}(t), k)$, $k>0$, from the definition of $\alpha$ we get

\begin{equation}\label{dpa}
\alpha(u_{n}(t)) =\text{Re} \sum_{l=1}^{\infty} \dfrac{(-1)^{l-1}}{l} \text{tr} \left\{  \left( (k - \partial_{x})^{- \frac 12} u_{n} (k + \partial_{x})^{- 1} \overline{u_{n}} (k - \partial_{x})^{-\frac12} \right)^{l}\right\}.
\end{equation}

Using \eqref{ConmTr}
\[
\text{tr}  \left( ( k - \partial_{x})^{- \frac 12}u_{n} (k + \partial_{x})^{- 1} \overline{u_{n}} (k - \partial_{x})^{- \frac 12} \right) =\text{tr}  \left( ( k - \partial_{x})^{- 1}u_{n} (k + \partial_{x})^{- 1} \overline{u_{n}}  \right), 
\]
 from \eqref{Pedrx2} and  \eqref{dpa}, we get

\begin{equation}\label{dpf}
\alpha(u_n(t)) = 2kc\int_{\R}\dfrac{|\widehat{u}_{n}(\xi, t)|^2}{4k^2 + \xi^2}\,d \xi +  \text{Re}\sum_{l=2}^{\infty} \dfrac{(-1)^{l-1}}{l} \text{tr}\left\{ \left[\,M^{-}u_{n}{ M^{+}}M^+\overline{u}_{n}M^-\,\right ]^{l}\,\right\}
\end{equation}
where $M^{\mp}=(k\mp\partial_x)^{-\frac12}$,
from \eqref{dpf} and \eqref{ignorm}, we get

\begin{equation}\label{dpg}
\left |\,\alpha(u_n(t)) - 2kc\int_{\R}\dfrac{|\widehat{u}_{n}(\xi, t)|^2}{4k^2 + \xi^2}\,d \xi\,\right | \leq \sum_{l=2}^{\infty}\dfrac{1}{l}\,||M^-u_{n} M^+ ||^{2 l}.
\end{equation}

However of Lemma \ref{lemaPedr},

\begin{equation}\label{dda}
||M^- u_{n} M^+ ||^{2 l}\, \lesssim\, \left [ \int_{\R} \dfrac{|\widehat{u}_{n}(\xi, t)|^2}{|k|+|\xi| }\,d \xi \right ]^{l}.
\end{equation}
then from \eqref{dpg} and \eqref{dda}, we get

\begin{equation}\label{ffa}
\left |\,\alpha(u_n(t)) - 2kc\int_{\R}\dfrac{|\widehat{u}_{n}(\xi, t)|^2}{4k^2 + \xi^2}\,d \xi\,\right | \lesssim \sum_{l=2}^{\infty}  \left [ \int_{\R} <\xi>^{-1}|\widehat{u}_{n}(\xi, t)|^2\, d \xi \right ]^{l}.
\end{equation} 

Using \eqref{ign}, making a change of variables in \eqref{ffa}, we get

\begin{equation}\label{ffb}
\left |\,\alpha(u_n(t)) - 2kc\int_{\R}\dfrac{|\widehat{u}_{n}(\xi, t)|^2}{4k^2 + \xi^2}\,d \xi\,\right | \,\lesssim\, \sum_{l=2}^{\infty}  \left [ \int_{\R} <\xi - n>^{-1}|\widehat{u}(\xi, t)|^2\, d \xi \right ]^{l}.
\end{equation} 

For each $j\in\mathbb{Z}$ consider the interval
\[
I_{j} = [ j - \frac 12, j + \frac 12 [
\]

Next, we will estimate the right side of \eqref{ffb}, when time $t = 0$. In this sense, we olserve that

\begin{equation}\label{gga}
\int_{\R} <\xi - n>^{-1}|\widehat{u}(\xi, 0)|^2\, d \xi \;\sim\;\sum_{j\in\mathbb{Z}}< j - n>^{- 1 }\int_{I_{j}}\,|\widehat{u}(\xi, 0)|^2\,d \xi.
\end{equation}

However, if $p \geq 2$ and $1<q \leq \infty$ such that $\frac{2}{p} + \frac{1}{q} = 1$,  using  H\"older's inequality, we get
\begin{equation}\label{ggb}
\begin{split}
\sum_{j\in\mathbb{Z}}< j - n>^{- 1 }\|\widehat{u}(\xi, 0)\|_{L^2(I_j)}^2 &\lesssim
 \|<j-n>^{- 1}\|_{\ell_j^q}\,\| \,\|\widehat{u}(\xi, 0)\|_{L^2(I_j)}^2 \|_{\ell_j^{\frac{ p}2}(\Z)}\\
&\lesssim \|\, \|\widehat{u}(\xi, 0)\|_{L^2(I_j)} \|_{\ell_j^{p}(\Z)}^2\, \sim \,\|\,u(0)\|_{M^{2,p}}^2
\end{split}
\end{equation}

uniformly in $n\in \Z$.

By continuity of the solution operator, there exists a neighborhood $ I_{\lambda} = ]- \lambda, \lambda [ $  around $t=0$  such that

\begin{equation}\label{mma}
\| \,u(t)\,\|_{M^{2,p}}\, \leq\,\epsilon\ll 1.
\end{equation}

Applying similar ideas to obtain \eqref{ggb} , we have

\begin{equation}\label{mmb}
r_{n}(t)\,:=\,\int_{\R} <\xi - n>^{-1}|\widehat{u}(\xi, t)|^2\, d \xi\,\lesssim\,\| u(t) \|_{M^{2,p}}^2.
\end{equation}

Note that  $\sum_{2}^{\infty}\,( r_{n}(t) )^{\textit{l}}$  it is a geometric series of ratio $ 0 \leq r_{n}(t) \leq \epsilon \ll 1$ and $t\in I_{\lambda}$ fixed any. 
Then

 From \eqref{ffb}, \eqref{mmb}, we have

\begin{equation}\label{mmg}
\left |\,\alpha(u_n(t)) - 2kc\int_{\R}\dfrac{|\widehat{u}_{n}(\xi, t)|^2}{4k^2 + \xi^2}\,d \xi\,\right | \,\lesssim\,\sum_{l=2}^{\infty}\,( r_{n}(t) )^{\textit{l}}\,=\, \dfrac{( r_{n}(t) )^2}{1-r_n(t)}\,\lesssim\, ( r_{n}(t) )^2,
\end{equation}
for all $n\in\mathbb{Z},\;t\in I_{\lambda}$. Thus using Young's inequality to series, considering $\frac1p+1=\frac1q+\frac2p$ ($q=\frac{p}{p-1}$), we have

\begin {equation}\label{mmj}
\begin{split}
\left\|\alpha(u_{n}(t)) - 2kc\int_{\R}\dfrac{|\widehat{u}_{n}(\xi, t)|^2}{4k^2 + \xi^2}\,d \xi\,\right\|_{\ell_n^{p/2}}&\,\lesssim\,\left\| \int_{\R} <\xi - n>^{-1}|\widehat{u}(\xi, t)|^2\, d \xi \right\|_{\ell_{n}^{p}}^{2}\\
&\,\lesssim\,\left\| \sum_{j\in\mathbb{Z}}< n - j>^{- 1 }\int_{I_{j}}\,|\widehat{u}(\xi, t)|^2\,d \xi.\right\|_{\ell_{n}^{p}}^{2}\\
&\,\lesssim\,\|\,\| \widehat{u}(\xi, t) \|_{L_{\xi}^{2}(I_j)}^{2}\,\|_{\ell_{j}^{p/2}}^{2}\\
&\, \sim\, \| u(t)\|_{M^{2,p}}^{4}
\end{split}
\end{equation}
evenly for all $t\in I_{\lambda}$.


On the other hand, changing the variable $ \mu = \xi - n $, we get

\begin{equation}\label{essa}
\,\left[ \sum_{n\in\mathbb{Z}} \| <\xi - n>^{- 1}\widehat{u}(\xi, t)\|_{L_{\xi}^{2}}^{p} \right]^{2/p}=\, \left\| \int_{\R}\dfrac{|\widehat{u}_n(\mu, t)|^2}{\langle \mu \rangle^2}\,d \mu \right\|_{\ell_{n}^{p/2}(\mathbb{Z})}\sim \| u(t) \|_{M^{2,p}}^{2}.
\end{equation}

From \eqref{mmj} and \eqref{essa}, we get

\begin{equation}\label{essd}
\left\|2kc\int_{\R}\dfrac{|\widehat{u}_{n}(\xi, t)|^2}{4k^2 + \xi^2}\,d \xi\right\|_{\ell_{n}^{p/2}} \,\sim_k\| u(t) \|_{M^{2,p}}^{2}\,\lesssim\,\| \alpha(u_{n}(t)) \|_{\ell_{n}^{p/2}} + \| u(t) \|_{M^{2,p}}^{4}.
\end{equation}

Furthermore by \eqref{mmj} and \eqref{essd}

\begin{equation}\label{essf}
\| \alpha(u_{n}(t)) \|_{\ell_{n}^{p/2}} \,\lesssim\,\| u(t)\|_{M^{2,p}}^{4} + \| u(t) \|_{M^{2,p}}^{2}.
\end{equation}

Note that

\begin{equation}\label{essh}
  \| u(t) \|_{M^{2,p}}^{2} \,\lesssim\, \| \alpha(u_{n}(t)) \|_{\ell_{n}^{p/2}} + \| u(t)\|_{M^{2,p}}^{4}.
\end{equation}

From \eqref{essf}, \eqref{essh} and given that $\alpha$  is conserved, we get

\begin{equation}\label{essk}
\| u(t) \|_{M^{2,p}}^{2} \,\lesssim\,\| u(0)\|_{M^{2,p}}^{4} + \| u(0) \|_{M^{2,p}}^{2} + \|u(t) \|_{M^{2,p}}^{4}.
\end{equation} 


for all $ t\in I_{\lambda}$.

From \eqref{mma} and the continuity, we get

\begin{equation}\label{essq}
|| u(t) ||_{M^{2,p}}^{2} \,\lesssim\,\| u(0) \|_{M^{2,p}}^{2}.
\end{equation}
evenly for all $t\in\R $.
Finally we will consider the general case, let $u_{\lambda}(x,t)=\lambda^{-1}u(\lambda^{-1}x, \lambda^{-3}t)$, then $v:=u_{\lambda}$ is a solution of
\begin{equation}
v_t+ia\lambda^{-1}v_{xx}+bv_{xxx}=2ia\lambda^{-1}|v|^2v+6b|v|^2v_x,
\end{equation}
 using the inequality $(c_1+ \cdots c_n)^{a} \leq n^{a-1} (c_1^a+ \cdots c_n^a)$, $c_j\geq0$, $j=1, \dots n$, we have
\begin {equation}\label{x0mmj}
\begin{split}
\|u_{\lambda}(\cdot, 0)|\|_{M^{2,p}}\sim&\lambda^{-1/2}\|\, \|\widehat{u}(\cdot,0)\|_{L^2(J_{\lambda,n} )}\,\|_{\ell^p_n},\qquad J_{\lambda,n}=\lambda I_n=[\lambda n-\frac\lambda2, n\lambda+\frac\lambda2]\\
=&\lambda^{-1/2}\left(\sum_{n\in \Z} \left(\sum_{j=0}^{\lambda-1}   \|\widehat{u}(\cdot,0)\|^2_{L^2(J_{\lambda,n,j})}\right)^{p/2}\right)^{1/p}\\
\leq& \lambda^{-1/p}\|u(\cdot,0)\|_{M^{2,p}}
\end{split}
\end{equation}
where $J_{\lambda,n,j}=[\lambda n-\frac\lambda2+j, n\lambda-\frac\lambda2+j+1]$, $n,\lambda \in \Z^+$, $0\leq j\leq \lambda -1$, hence choosing $\lambda \in Z^+$ such that
$$
\lambda \sim (1+\|u(\cdot,0)\|_{M^{2,p}})^p
$$
we obtain
$$
\|u_{\lambda}(\cdot, 0)\|_{M^{2,p}} <\epsilon \ll 1,
$$
 by the small data case presented above follows that.
\begin{equation}\label{Xav1}
\|u_{\lambda}(\cdot, t)|\|_{M^{2,p}} \lesssim \|u_{\lambda}(\cdot, 0)\|_{M^{2,p}}
\end{equation}
for all $t\in \R$.
By scaling, \eqref{Xav1} and \eqref{x0mmj} holds
\begin {equation}\label{x1mmj}
\begin{split}
\|u(t)\|_{M^{2,p}}\leq \lambda^{1/2}\|u_{\lambda}(\cdot,\lambda^3 t)\|_{M^{2,p}} \lesssim\lambda^{1/2}\|u_{\lambda}(\cdot,0)\|_{M^{2,p}} \lesssim\lambda^{1/2-1/p}\|u(\cdot,0)\|_{M^{2,p}}.
\end{split}
\end{equation}
\end{proof}
 This inequality proves the proposition.
\begin{proposition}\label{gwps}
Let $2\leq p<\infty$, $0 \leq s< 1- \frac1p$ and $u\in S(\R)$ a solution to \eqref{modeloHONSE}. Then, there exists $C=C(p)$ positive such that
\begin{equation}
\|u(t)\|_{M^{2,p}_s(\R)}\leq C\left( 1+\|u(0)\|_{M^{2,p}_s(\R)}  \right)^{p/2-1}\|u(0)\|_{M^{2,p}_s(\R)},
\end{equation}
for any $t\in \R$.
\end{proposition}
\begin{proof}
Is the same proof as the proof of Theorem B1 (i) in \cite{OW-20}.
\end{proof}

{\bf  Proof of Theorem \ref{main-th2}}

The proof of Theorem \ref{main-th2} in the case $0\leq s\leq 1-\frac1p$, follows from Propositions  \ref{gwpapr} and \ref{gwps} and in the case $s>1-\frac1p$ the proof is similar as in Section 3.6 of \cite{Oh-Wang}.






\begin{thebibliography}{99}

\bibitem{Agr-07} G. Agrawal, {\em Nonlinear Fiber Optics}, Fourth Edition, Elsevier Academic Press, Oxford (2007).

\bibitem{B-93} J. Bourgain, {\em Fourier transform restriction phenomena for certain lattice subsets and applications to nonlinear evolution equations I. Schr\" odinger equations}, 
Geom. Funct. Anal., {\bf 3} (2) (1993) 107--156.


\bibitem{CW-90} T. Cazenave, F.B. Weissler, {\em The Cauchy problem for the critical nonlinear Schr\"odinger equation}, Nonlinear Analysis TMA, {\bf 14} (1990), 807--836.

\bibitem{XC-04} X. Carvajal, {\em Local well-posedness for a higher order nonlinear Schr\"odinger equation in Sobolev spaces if negative indices}, Electronic J Diff Equations {\bf 2004} (2004) 1--10.

\bibitem{CA1} X. Carvajal,  {Ph.D. Thesis}, IMPA, (2002).

\bibitem{XC-06} X. Carvajal,  {\em Sharp global well-posedness for a higher order Schr\"odinger equation}, Jr. Fourier Anal Appl {\bf 12} (2006) 53--73.

\bibitem{CL1} X. Carvajal, F. Linares, { \em A higher order nonlinear
Schr\"odinger equation with variable coefficients},
Diff. Int. Equations {\bf 16} (2003), 1111--1130.

\bibitem{CP} X. Carvajal, M. Panthee, {\em Nonlinear schrödinger equations with the third order dispersion on modulation spaces}, Partial Differ. Equ. Appl. 3 (2022), no. 5, Paper No. 59, 21 pp. 

\bibitem{CKSTT} J. Colliander, M. Keel, G. Staffilani, H. Takaoka, T. Tao, {\em Sharp global well-posedness for periodic and nonperiodic KdV and mKdV}, J. Amer. Math. Soc. {\bf 16}, (2003), 705--749.

\bibitem{DT-21} A. Debussche, Y. Tsutsumi, {\em Quasi-invariance of Gaussian measures transported by the cubic NLS with third-order dispersion on $\T$}, J. Funct. Anal. {\bf 281} (2021),  109032, 23 pp.

\bibitem{[FLP]}G. Fonseca, F. Linares, G. Ponce, {Global
well-posedness for the modified Korteweg-de Vries equation}, Comm.
Partial Differential Equations, {\bf 24} (1999), 683--705.



\bibitem{Feich-1} H. G. Feichtinger and K. H. Gr\"ochenig, {\em Banach spaces related to integrable group representations
and their atomic decompositions I}, J. Func. Anal., {\bf 86} (1989) 307--340.

\bibitem{Feich-2} H. G. Feichtinger and K. H. Gr\"ochenig, {\em Banach spaces related to integrable group representations
and their atomic decompositions II}, Monatsh. Math., 108 (1989) 129--148.

\bibitem{Gr-04} A. Gr\"unrock, {\em An improved local well-posedness result for the modified KdV equation}, Int. Math. Res. Not., 2004 (2004), 3287--3308.

\bibitem{GV-09}  A. Gr\"unrock and L. Vega, {\em Local well-posedness for the modified KdV equation in almost critical $\widehat{H^r_s}$ -spaces}, Trans. Amer. Math. Soc., {\bf 361} (2009), 5681--5694.

\bibitem{[H-K]} A. Hasegawa, Y. Kodama, {Nonlinear pulse propagation
in a monomode dielectric guide}, IEEE J. of Quantum Electronics, {\bf 23} 
 (1987), 510--524.
 
\bibitem{HK-81} A. Hasegawa and Y. Kodama, {\em Signal transmission by optical solitons in monomode Fiber}, Proc. IEEE {\bf 69} (1981) 1145--1150.

\bibitem{[KPV1]} C. E. Kenig, G. Ponce, L. Vega, {Well-posedness and
scattering results for the generalized Korteweg-de Vries equation via
the contraction principle}, Comm. Pure Appl. Math., {\bf 46} (1993),
527--620.

\bibitem{[KVZ]} R. Killip, M. Visan, X. Zhang, {Low regularity conservation laws for integrable PDE}, Geom. Funct. Anal., {\bf 28} (2018),
1062--1090.

\bibitem{[Ko]} Y. Kodama, {Optical solitons in a monomode fiber}, J.
Statistical Phys., {\bf 39}  (1985), 597--614.

\bibitem{[C2]} C. Laurey, {The Cauchy problem for a third order
nonlinear Schr\"odinger equation}, Nonlinear Anal., Theory, Methods
and Appl., {\bf 29} (1997), 121--158.

\bibitem{MT-18} T. Miyaji, Y. Tsutsumi, {\em Local well-posedness of the NLS equation with third order dispersion in negative Sobolev spaces},  Differential Integral Equations {\bf 31} (2018) 111--132.

\bibitem{MT-17} T. Miyaji, Y. Tsutsumi, {\em Existence of global solutions and global attractor for the third order Lugiato-Lefever equation on $\T$}, Ann. I. H. Poincar\'e-AN {\bf 34} (2017) 1707--1725.

\bibitem{OTT-19} T. Oh, Y. Tsutsumi, N. Tzvetkov, {\em Quasi-invariant Gaussian measures for the cubic nonlinear Schr\"odinger equation with third-order dispersion}, C. R.Acad.Sci.Paris,Ser.I {\bf 357} (2019)366--381.

\bibitem{Oh-Wang} T. Oh, Y. Wang, {\em On the global well-posedness of the modified KdV equation in modulation spaces}, Discrete and Continuous Dynamical Systems {\bf 41} (2021) 2971--2992.

\bibitem{OW-20} T. Oh, Y. Wang, {\em Global well-posedness of the one-dimensional cubic nonlinear Schr\"odinger equation in almost critical spaces}, J. Differential Equations {\bf 269} (2020) 612--640.

\bibitem{Oikawa-93} M. Oikawa, {\em Effect of the third-order dispersion on the nonlinear Schr\'odinger equation}, J. Phys. Soc. Japan, {\bf 62} (1993) 2324--2333.

\bibitem{[G]}G. Staffilani, {On the generalized Korteweg-de Vries-type
equations}, Diff. Int. Equations. {\bf 10} (1997), 777--796.

\bibitem{Takaoka} H. Takaoka, {\em Well-posedness for the higher order nonlinear Schr\"odinger equation},  Adv. Math. Sci. Appl.{\bf 10} (2000) 149--171.


\bibitem{Tsutsumi-18} Y. Tsutsumi, {\em Well-Posedness and Smoothing Effect
for Nonlinear Dispersive Equations}, available in {\href{https://krieger.jhu.edu/math/wp-content/uploads/sites/62/2018/03/jami2018lecture2abstract1.pdf}{https://krieger.jhu.edu/math/wp-content/uploads/sites/62/2018/03/jami2018lecture2abstract1.pdf}}, 2018.

\bibitem{Ts-87} Y. Tsutsumi, {\em $L^2$ solutions for nonlinear Schr\"odinger equations and nonlinear groups}, Funk. Ekva. {\bf 30} (1987), 115--125.


\end{thebibliography}
\end{document}